\documentclass{article}
\usepackage{latexsym}
\usepackage{amssymb}
\usepackage{amsmath}
\usepackage{color}
\usepackage{graphicx}
\usepackage{amsthm}
\usepackage{indentfirst}
\allowdisplaybreaks

\newtheorem{theorem}{\color{black}\indent Theorem}[section]
\newtheorem{lemma}{\color{black}\indent Lemma}[section]

\newtheorem{remark}{\color{black}\indent Remark}[section]

\begin{document}
\large
\title{Lifespan of Solutions to Wave Equations on de Sitter Spacetime}
\author{Weiping Yan
{\sc}
\thanks{
School of Mathematics, Xiamen University, Xiamen,
361000, P.R. China.~~~ Email address: yan8441@126.com;yanwp@xmu.edu.cn}}
\date{}

\maketitle
\par
{\bf Abstract.}
In this paper, we consider the finite time blow up of solutions for the following two kinds of nonlinear wave equations on de Sitter spacetime
\begin{eqnarray*}
&&\square_g=F(u),\\
&&\square_g=F(\partial_tu,\nabla u).
\end{eqnarray*}
This proof is based on a new blow up criterion, which generalize the blow up criterion in Sideris \cite{Sider}. Furthermore, we give the lifespan estimate of solutions for the problems.

\section{Introduction and Main result}\setcounter{equation}{0}
The general purpose of this paper is to initiate the study of the finite time blow up of solutions for the following nonlinear wave equation on de Sitter Spacetime:
\begin{eqnarray}\label{E1-1}
\square_g=F(u,\partial_tu,\nabla u),
\end{eqnarray}
where $\square_g=\frac{1}{\sqrt{|g|}}\frac{\partial}{\partial x^i}(\sqrt{|g|}g^{ik}\frac{\partial }{\partial x^k})$ is the d'Alembertian in the de Sitter metric.

In the models proposed by Einstein \cite{E} and de Sitter \cite{Des} the universe is assumed
to be spatially isotropic, as well as a static system. This means that one can choose a suitable system of coordinates
in which the line element has the static and
spherically symmetric form
\begin{eqnarray*}
ds^2=-(1-\frac{M_{bh}}{r}-\frac{\Lambda r^2}{3})c^2dt^2+(1-\frac{2M_{bh}}{r}-\frac{\Lambda r^2}{3})^{-1}dr^2
+r^2(d\theta^2+\sin^2\theta d\phi^2),
\end{eqnarray*}
where $M_{bh}$ has a meaning of the ``mass of the black hole'', $\Lambda$ is the cosmological constant.
Even a small value of $\Lambda$ could have drastic effects on the evolution of the universe.
The corresponding metric with this line element is called the Schwarzschild-de Sitter metric, which leads to a solution of the Einstein equation with the cosmological constant
\begin{eqnarray*}
R_{\mu\nu}-\frac{1}{2}g_{\mu\nu}R=-8\pi GT_{\mu\nu}-\Lambda g_{\mu\nu}.
\end{eqnarray*}
In particularly, we ignore the influence of the black hole, $i.e.$ $M_{bh}=0$, the line element in the de Sitter spacetime has the form
\begin{eqnarray*}
ds^2=-(1-\frac{r^2}{R^2})c^2dt^2+(1-\frac{r^2}{R^2})^{-1}dr^2
+r^2(d\theta^2+\sin^2\theta d\phi^2).
\end{eqnarray*}
The Lama$\hat{i}$tre-Robertson transformation \cite{Mo} leads to the following form for the line
element
\begin{eqnarray*}
ds^2=-c^2dt'^2+e^{2ct'R^{-1}}(dx'^2+dy'^2+dz'^2),
\end{eqnarray*}
where $R$ is the ``radius'' of the universe. The new coordinates $x'$, $y'$, $z'$ and $t'$ can take all values from $-\infty$ to $+\infty$.
Furthermore, in the Robertson-Walker spacetime \cite{Hw}, one can choose system of coordinates in which the metric has the form
\begin{eqnarray*}
ds^2=-dt^2+S^2(t)(dx^2+dy^2+dz^2),
\end{eqnarray*}
with an appropriate scale factor $S(t)$. In particular, the metric in de Sitter spacetime corresponds to the cosmic scale factor $S(t)=S(0)e^{\sqrt{\frac{\Lambda}{3}}t}$, where cosmological constant $\Lambda$ is a positive constant, which produces models with an exponentially accelerated expansion.

The Cauchy problem of nonlinear wave equations on de Sitter spacetime corresponds to nonlinear wave equations with time-dependent coefficients.
Colombini and Spagnolo \cite{Colom} studied the following Cauchy problem of one dimensional wave equation with time-dependent coefficients
\begin{eqnarray}\label{El}
u_{tt}-a(t)u_{xx}=0,~~(t,x)\in[0,T]\times\textbf{R},
\end{eqnarray}
where $a(t)$ is a nonnegative smooth function, which oscillates an infinite number of times. By imposing some assumptions, they showed that (\ref{El})
is not well-posed in $\mathbb{C}^{\infty}$. After that, Colombini, De Giorgi and Spagnolo \cite{Colom1} showed that
any Cauchy problem of (\ref{El}) is well-posed in the space of periodic real analytic functions.
In particularly, let $a(t)=e^{-2Ht}$, the linear wave equation with time-dependent coefficients
\begin{eqnarray*}\label{El1}
u_{tt}-e^{-2Ht}\triangle u=0
\end{eqnarray*}
is strictly hyperbolic. The speed of propagation is equal to $e^{-Ht},~~\forall t\in\mathbb{R}$.
The resolving operator for the corresponding Cauchy problem can be rewritten as a sum of the Fourier integral operators with the amplitudes given in terms of the Bessel functions and in terms of confluent hypergeometric functions. We refer to \cite{Ba,Ch,Yag1,Yag2,Yag3,Yag4} for more results of the semilinear Klein-Gordon equation on de Sitter spacetime.
Recently, Galstian and Yagdjian \cite{Yag5} obtained that the global existence for the semilinear Klein-Gordon equation in FLRW spacetime.

The Strauss' conjecture concerns with the existence or nonexistence of global solutions to
the semilinear wave equation
\begin{eqnarray}\label{E1-00}
\partial_{tt}u-\sum_{i=1}^n\partial_i^2u=|u|^p,~~(t,x)\in\mathbb{R}^+\times\mathbb{R}^n,~n\geq2.
\end{eqnarray}
Above problem was firstly studied by John \cite{John}
with $n=3$. More precisely, he showed that the semilinear wave equation (\ref{E1-00}) with small initial data has global solutions for the exponent $p>1+\sqrt{2}$. Meanwhile, he proved that the finite time blow up of solutions for the exponent $p<1+\sqrt{2}$, where the initial data is not zero. Strauss conjectured that the existence or nonexistence of global solutions to equation (\ref{E1-00}) with spatial dimension $n\geq2$ for the exponent $p\in(p_c(n),\infty)$ or $p\in(1,p_c(n)]$, where $p_c(n)$ is the positive root of the quadratic equation
\begin{eqnarray*}
(n-1)p^2-(n+1)p-2=0.
\end{eqnarray*}
After that, there are many results concerning with this conjecture. We give a brief summary here.
For the global existence of solutions of (\ref{E1-1}), we refer to Glassey \cite{Glassey1} for $n=2$,
Lindblad and Sogge \cite{Lin2} for $n\leq8$,
and Sideris \cite{Sider} for $n\geq4$. Georgiev, Lindblad, Sogge \cite{Geo} for $n\geq4$ and $p_c<p\leq\frac{n+3}{n-1}$. For the finite time blow up of solutions of (\ref{E1-1}), one can see Glassey \cite{Glassey2} for $n=2$ and Sideris \cite{Sider} for $n\geq4$, Schaeffer \cite{Schae} for $n=2,3$ and the critical case $p=p_c(n)$, Yordanov and Zhang \cite{Yorda2} and Zhou \cite{zhou0} for $n\geq4$, Takamura and Wakasa \cite{Hiro} and Zhou and Han \cite{zhou2} for $n\geq2$, where the sharp upper bound of the lifespan of solutions by using different method. The lifespan $T(\epsilon)$ of solutions to (\ref{E1-00}) is the largest value such that solutions exist for
$(t,x)\in(0,T(\epsilon))\times\mathbb{R}^n$.
To the best knowledge of authors, there is few result concerning with the Strauss's conjecture on cosmological spacetime except the work of
Lindblad, Metcalfe, Sogge, Tohaneanu and Wang \cite{Lin1}. They proved that the global existence of solutions for the semilinear wave equation (\ref{E1-00}) on Kerr black hole backgrounds.

Let the potential function
\begin{equation}\label{rr1}
F(u,\partial_tu,\nabla u)=-|u|^p.
\end{equation}
In the de Sitter spacetime, the semilinear wave equation (\ref{E1-1}) can be rewritten explicitly in coordinates as
\begin{eqnarray}\label{E1-1R1}
u_{tt}-e^{-2Ht}\triangle u=e^{-\frac{n}{2}(p-1)Ht}|u|^p,
\end{eqnarray}
where $(t,x)\in(0,\infty)\times\mathbb{R}^n$, $p>1$, $\triangle$ is the Laplace operator on the flat metric and $H=\sqrt{\frac{\Lambda}{3}}$ is the Hubble constant.

In fact, the rigorous derivation of equation (\ref{E1-1R1}) is from the semilinear Klein-Gordon equation with mass $m=\frac{nH}{2}$. The equation (\ref{E1-1R1}) is the equation of graviton. One can see \cite{Yag2,Yag3} for more details.

Assume that the initial data with compactly supported
\begin{eqnarray}\label{E1-2}
u(0,x)=\epsilon f(x),~~u_t(0,x)=\epsilon g(x),~~x\in\mathbb{R}^n
\end{eqnarray}
satisfies
\begin{eqnarray}\label{E1-3}
f(x),~~g(x)\geq0,~~f(x)=g(x)=0~for~|x|>1,~~g(x)\not\equiv0,
\end{eqnarray}
where $\epsilon$ is a small postive constant, $f(x),g(x)\in\mathbb{C}_0^{\infty}(\mathbb{R}^n)$.

\begin{theorem}
Let $f,g$ be smooth functions with compact support $f,g\in\mathbb{C}_0^{\infty}$, and $F$ be the form of (\ref{rr1}).
Assume that (\ref{E1-2}), (\ref{E1-3}) and $0<H<2$ hold, and problem (\ref{E1-1}) has a solution $(u,u_t)\in\mathbb{C}([0,T),\mathbb{H}^1(\mathbb{R}^n)\times\mathbb{L}^{2}(\mathbb{R}^n))$ such that
\begin{eqnarray*}
\textbf{supp}(u,u_t)\subset\{(t,x):|x|\leq1+t\},~~n\geq2.
\end{eqnarray*}
If $1<p< p_c(n)$, then the solution $u(t,x)$ will blow up in finite time, that is $T<\infty$. Moreover, we have the following estimate for the lifespan $T(\epsilon)$ of solutions for (\ref{E1-1}) such that
\begin{eqnarray*}
T(\epsilon)\leq C_0\epsilon^{-\frac{p-1}{(p-1)[1-(n-1)\frac{p}{2}]+2}},
\end{eqnarray*}
where $2^{-\frac{(p-1)[1-(n-1)\frac{p}{2}]+2}{p-1}}\leq\epsilon\leq1$ and $C_0$ is a positive constant independent of $\epsilon$.
\end{theorem}

\begin{remark}
We notice that the semilinear Klein-Gordon equation in de Sitter spacetime admits the finite time blow up of solutions $u\in\mathbb{C}([0,T);\mathbb{L}^q(\mathbb{R}^n))$ with $q\in[2,\infty)$ and $p>1$, one can see \cite{Yag3} for more details. Comparing with semilinear wave equation in de Sitter spacetime ($i.e.$ Theorem 1.1), we find that the main difference of two finite time blow up results is the exponent $p\in (1,p_c(n))$, where $p_c(n)$ is the positive root of $(n-1)p^2-(n+1)p-2=0$. We think that this result also holds for the semilinear wave equation in FLRW spacetime, i.e. an analog of the Strauss conjecture. Here we do not intend to give this proof of details.
\end{remark}

The Glassey \cite{Glassey3} (or see \cite{Ramm} and \cite{Schae1}) made the conjecture that the exponent $p_c'=1+\frac{2}{n-1}$ is
the critical exponent for the global existence of
wave equation
\begin{eqnarray}\label{E1-001}
\partial_{tt}u-\sum_{i=1}^n\partial_i^2u=|\partial_tu|^p+|\nabla u|^p,~~(t,x)\in\mathbb{R}^+\times\mathbb{R}^n,~n\geq2.
\end{eqnarray}
Glassey's conjecture was verified by Sideris \cite{Sider1} for $n=3$ with radial data. 
Hidano and Tsutaya \cite{Hi} or Tzvetkov \cite{Tz} verified this for $n=2,3$ with general data.
After that, Hidano, Wang and Yokovama \cite{Hi1} proved this for $n\geq2$ with radially symmetric data. Recently, Wang \cite{wang} obtained this for $n=3$ on asymptotically flat manifolds.
For the finite time blow up of solutions to (\ref{E1-001}), there is only a result of Zhou \cite{zhou4} for $n\geq4$, meanwhile, he gave an explicit upper bound to the lifespan of solutions.

The second aim of this paper is to study an analog of the Glassey conjecture on de Sitter spacetime.
Let the potential function
\begin{equation}\label{Er1}
F(u,\partial_tu,\nabla u)=-(|\partial_tu|^p+|\nabla u|^p).
\end{equation}
Then in the de Sitter spacetime,
the semilinear wave equation (\ref{E1-001}) can be rewritten in coordinates as
\begin{eqnarray}\label{E1-002}
u_{tt}-e^{-2Ht}\triangle u=e^{-\frac{n}{2}(p-1)Ht}(|\partial_tu|^p+|\nabla u|^p).
\end{eqnarray}

Now we establish the following blow up result for (\ref{E1-002}). 

\begin{theorem}
Let $f,g$ be smooth functions with compact support $f,g\in\mathbb{C}_0^{\infty}$, and $F$ be the form of (\ref{Er1}).
Assume that (\ref{E1-2}), (\ref{E1-3}) and $H>0$ hold, and problem (\ref{E1-002}) has a solution $(u,u_t)\in\mathbb{C}([0,T),\mathbb{H}^1(\mathbb{R}^n)\times\mathbb{L}^{2}(\mathbb{R}^n))$ such that
\begin{eqnarray*}
\textbf{supp}(u,u_t)\subset\{(t,x):|x|\leq1+t\},~~n\geq2.
\end{eqnarray*}
If $1<p< p_c'(n)=1+\frac{2}{n-1}$, then the solution $u(t,x)$ will blow up in finite time, that is $T<\infty$. Moreover, we have the following estimate for the lifespan $T(\epsilon)$ of solutions for (\ref{E1-002}) such that
\begin{eqnarray*}
T(\epsilon)\leq \epsilon^{-(p-1)},
\end{eqnarray*}
where $2^{-\frac{1}{p-1}}\leq\epsilon\leq1$ and $C_0$ is a positive constant independent of $\epsilon$.
\end{theorem}

The organization of this paper is as follows. In Section 2, two new blow up criterions are given. Section 3 is devoted to prove the finite time blow up of solutions for the semilinear wave equation (\ref{E1-1}) on de Sitter spacetime. Furthermore, the estimate of the lifespan of solutions is obtained. In the last section, the proof of Theorem 1.2 is given.

\section{Some new blow up criterions}\setcounter{equation}{0}
In this section, we give two new blow up criterions for ordinary differential inequality.
\begin{lemma}
Let $p > 1$ and $b_1-a_1(p-1)<2$. Assume that $G\in\mathbb{C}^2([0, T ))$ satisfies
\begin{eqnarray}\label{E2-1RR}
&&G(t)\geq Kt^{a_1}~~~~for~t\geq T_0,\\
\label{E2-2RR}
&&G''(t)\geq Ae^{-b_1(t+R)}|G(t)|^p~~~~for~t>0,\\
\label{E2-3RR}
&&G(0)>0,~~G'(0)>0,
\end{eqnarray}
where $K$, $T_0$, $A$ and $R$ denote positive constants with $T_0\geq R$.

Then $T$ must satisfy that $T\leq 2T_1$ provided that $K\geq K_0$, where $K_0$ is a fixed positive constant.

Furthermore, we have the life span $T(\epsilon)$ of $F(t)$, which satisfies
\begin{eqnarray}\label{E2-ERR}
T(\epsilon)\leq C_0\epsilon^{-\frac{(p-1)}{(p-1)a_1-b_1+2}},
\end{eqnarray}
where $2^{-\frac{(p-1)a_1-b_1+2}{p-1}}\leq\epsilon\leq1$, $b_1>0$ and $C_0$ is a positive constant depending on $A$ and $R$ but independent of $\epsilon$.
\end{lemma}
\begin{proof}
Arguing by contradiction, assume that we have $T>2T_1$. Inequality (\ref{E2-2RR}) implies that
\begin{eqnarray*}
G''(t)>0,~~\forall t>0,
\end{eqnarray*}
which combining with (\ref{E2-3RR}) gives that
\begin{eqnarray}\label{E2-6RR}
G'(t)\geq G'(0)>0,~~G(t)\geq G'(0)t+G(0)\geq G(0)>0,~\forall t>0.
\end{eqnarray}
Multiplying (\ref{E2-2RR}) by $G'(t)$ and integrating it over $[0,t]$, we have
\begin{eqnarray*}
\frac{1}{2}G'(t)^2&\geq& A\int_0^te^{-b_1(s+R)}G^p(s)dG(s)+\frac{1}{2}G'(0)^2\nonumber\\
&>&\frac{A}{(p+1)}e^{-b_1(t+R)}(G(t)^{p+1}-G(0)^{p+1})\nonumber\\
&\geq&\frac{A}{(p+1)}e^{-b_1(t+R)}G(t)^{p}(G(t)-G(0)),~~\forall t>0.
\end{eqnarray*}
Restricting the time interval to $t\geq\frac{G(0)}{G'(0)}$ and making use of (\ref{E2-6RR}), we have
\begin{eqnarray*}
\frac{1}{2}G(t)-G(0)\geq\frac{1}{2}(G'(0)t-G(0))>0.
\end{eqnarray*}
Thus we get
\begin{eqnarray*}
G(t)'>(\frac{A}{p+1})^{\frac{1}{2}}e^{-\frac{b_1}{2}(t+R)}G^{\frac{p+1}{2}}(t),~~\forall t\geq\frac{G(0)}{G'(0)},
\end{eqnarray*}
which combining with (\ref{E2-1RR}) gives
\begin{eqnarray}\label{E2-7RR}
\frac{G(t)'}{G(t)^{1+\delta}}&>&(\frac{A}{p+1})^{\frac{1}{2}}e^{-\frac{b_1}{2}(t+R)}G^{\frac{p-1}{2}-\delta}(t)\nonumber\\
&\geq&(\frac{A}{p+1})^{\frac{1}{2}}K^{\frac{p-1}{2}-\delta}e^{-\frac{b_1}{2}(t+R)}t^{a_1(-\delta+\frac{p-1}{2})},~~\forall t\geq T_1\geq R,
\end{eqnarray}
where $\delta\in(0,\frac{p-1}{2})$ is a fixed positive constant.

Then integrating (\ref{E2-7RR}) over $[2T_1,t]$,
\begin{eqnarray*}
\delta^{-1}(G(2T_1)^{-\delta}-G(t)^{-\delta})>(\frac{A}{p+1})^{\frac{1}{2}}K^{\frac{p-1}{2}-\delta}\int_{2T_1}^te^{-\frac{b_1}{2}(s+R)}s^{a_1(-\delta+\frac{p-1}{2})}ds,
\end{eqnarray*}
which gives that
\begin{eqnarray}\label{E2-8RR}
\delta^{-1}G(2T_1)^{-\delta}>(\frac{A}{p+1})^{\frac{1}{2}}K^{\frac{p-1}{2}-\delta}\int_{2T_1}^te^{-\frac{b_1}{2}(s+R)}s^{a_1(-\delta+\frac{p-1}{2})}ds.
\end{eqnarray}
We define a function
\begin{eqnarray*}
\tilde{F}(t):=\int_{2T_1}^te^{-\frac{b_1}{2}(s+R)}s^{a_1(-\delta+\frac{p-1}{2})}ds,~~\forall~t> 2T_1.
\end{eqnarray*}
Note that $\delta\in(0,\frac{p-1}{2})$. It is easy to see that
\begin{eqnarray*}
\tilde{F}(t)>0,~~\forall~t> 2T_1,
\end{eqnarray*}
and
\begin{eqnarray*}
\frac{d\tilde{F}(t)}{dt}=e^{-\frac{b_1}{2}(t+R)}t^{a_1(-\delta+\frac{p-1}{2})}>0, ~~\forall~t> 2T_1.
\end{eqnarray*}
This implies that $\tilde{F}(t)$ is an unbounded increasing function in $(2T_1,+\infty)$. So $\tilde{F}(t)^{-1}$ is a bounded decreasing function in $(2T_1,+\infty)$, and
\begin{eqnarray}\label{rrr1}
\lim_{t\longrightarrow+\infty}\tilde{F}(t)^{-1}=0,~~\lim_{t\longrightarrow 2T_1^+}\tilde{F}(t)^{-1}=+\infty.
\end{eqnarray}
On the other hand, let $t=2T_1$ in (\ref{E2-1RR}), we get
\begin{eqnarray*}
G(T_1)\geq K(2T_1)^{a_1},
\end{eqnarray*}
which combining with (\ref{E2-8RR}) gives that
\begin{eqnarray*}
K^{-\delta}\geq[\frac{(2T_1)^{a_1}}{G(2T_1)}]^{\delta}>\delta(\frac{A}{p+1})^{\frac{1}{2}}(2T_1)^{a_1\delta}K^{\frac{p-1}{2}-\delta}\int_{2T_1}^te^{-\frac{b_1}{2}(s+R)}s^{a_1(-\delta+\frac{p-1}{2})}ds.
\end{eqnarray*}
This implies that
\begin{eqnarray}\label{rr6}
0<K<C_{a_1,A,p,\delta,T_1}(\int_{2T_1}^te^{-\frac{b_1}{2}(s+R)}s^{a_1(-\delta+\frac{p-1}{2})}ds)^{\frac{-2}{p-1}}=C_{a_1,A,p,\delta,T_1}\tilde{F}(t)^{\frac{-2}{p-1}},
\end{eqnarray}
where $C_{a_1,A,p,\delta,T_1}:=\delta^{\frac{-2}{p-1}}(\frac{A}{p+1})^{\frac{-1}{p-1}}T_1^{\frac{-2a_1\delta}{p-1}}$ is a positive constant.

Since $\tilde{F}(t)^{-1}\in(0,+\infty)$ is a decreasing function on $(2T_1,\infty)$, by (\ref{rrr1}), we know that there exists $t^*>2T_1$ such that
\begin{eqnarray*}
0<C_{a_1,A,p,\delta,T_1}\tilde{F}(t^*)^{\frac{-2}{p-1}}\leq K_0,
\end{eqnarray*}
which combining with (\ref{rr6}) gives that
\begin{eqnarray*}
K<K_0.
\end{eqnarray*}
This inequality contradicts to the choice of $K\geq K_0$. Therefore we conclude that $T\leq2T_1$.

In what follows, we prove the lifespan of $G(t)$. Let us make a translation
\begin{eqnarray}\label{rrr2}
&&\tau=t\epsilon^{\frac{(p-1)}{(p-1)a_1-b_1+2}},\nonumber\\
&&\mathcal{G}(\tau)=\epsilon^{\frac{(b_1-2)}{(p-1)a_1-b_1+2}}G(\tau\epsilon^{\frac{-(p-1)}{(p-1)a_1-b_1+2}}),
\end{eqnarray}
where $2^{-\frac{(p-1)a_1-b_1+2}{p-1}}\leq\epsilon\leq1$.

Note that $\epsilon^{-\frac{p-1}{(p-1)a_1-b_1+2}}-1\geq1$ and $\epsilon^{\frac{-b_1(p-1)}{(p-1)a_1-b_1+2}}\geq1$ with $b_1>0$.
Using (\ref{E2-2RR}) and (\ref{rrr2}), we derive
\begin{eqnarray*}\label{rrr3}
\mathcal{G}''(\tau)&=&\epsilon^{\frac{(b_1-2)-2(p-1)}{(p-1)a_1-b_1+2}}G''(\tau\epsilon^{\frac{-(p-1)}{(p-1)a_1-b_1+2}})\nonumber\\
&\geq&\epsilon^{\frac{(b_1-2)-2(p-1)}{(p-1)a_1-b_1+2}}Ae^{-b_1(\tau\epsilon^{\frac{-(p-1)}{(p-1)a_1-b_1+2}}+R)}|G(\tau\epsilon^{\frac{-(p-1)}{(p-1)a_1-b_1+2}})|^p\nonumber\\
&=&\epsilon^{\frac{-(b_1-2)(p-1)-2(p-1)}{(p-1)a_1-b_1+2}}Ae^{-b_1(\tau\epsilon^{\frac{-(p-1)}{(p-1)a_1-b_1+2}}+R)}|\mathcal{G}(\tau)|^p\nonumber\\
&=&\epsilon^{\frac{-b_1(p-1)}{(p-1)a_1-b_1+2}}e^{-b_1\tau(\epsilon^{\frac{-(p-1)}{(p-1)a_1-b_1+2}}-1)}Ae^{-b_1(\tau+R)}|\mathcal{G}(\tau)|^p\nonumber\\
&\geq&Ae^{-b_1(\tau+R)}|\mathcal{G}(\tau)|^p.
\end{eqnarray*}
On the other hand, by (\ref{E2-1RR}) and (\ref{rrr2}), we have
\begin{eqnarray*}\label{rrr4}
\mathcal{G}(\tau)&\geq&K\epsilon^{\frac{(b_1-2)}{(p-1)a_1-b_1+2}}\tau^{a_1}\epsilon^{\frac{-a_1(p-1)}{(p-1)a_1-b_1+2}}\nonumber\\
&=&K\epsilon^{-1}\tau^{a_1}\nonumber\\
&\geq&K\tau^{a_1}.
\end{eqnarray*}
Hence $\mathcal{G}(\tau)$ will blow up in finite time and the life span of $G(t)$ satisfies (\ref{E2-ERR}).
This completes the proof.
\end{proof}

\begin{lemma}
Let $p > 1$. Assume that $G\in\mathbb{C}^2([0, T ))$ satisfies
\begin{eqnarray}\label{E2-1}
&&G(t)\geq Ka(t)~~~~for~t\geq T_0,\\
\label{E2-2}
&&G''(t)\geq Ab^{-1}(t+R)|G(t)|^p~~~~for~t>0,\\
\label{E2-3}
&&G(0)>0,~~G'(0)>0,
\end{eqnarray}
where $K$, $T_0$, $A$ and $R$ denote positive constants with $T_0\geq R$, $a(t)$ and $b(t)$ are positive strictly increasing smooth functions and $b^{-\frac{1}{2}}(t+R)a^{\frac{p-1}{2}-\delta}(t)\in(0,\infty)$ is a strictly decreasing smooth function for $t>0$, and there exist $t^{**}>2T_1$ and a positive constant $K_0$ such that
\begin{eqnarray}\label{E2-4}
\delta^{\frac{2}{p-1}}(\frac{A}{p+1})^{\frac{1}{p-1}}K_0^{-1}a^{\frac{2\delta}{p-1}}(2T_1)\leq(\int_{2T_1}^{t^{**}}b(t+R)^{-\frac{1}{2}}a^{\frac{p-1}{2}-\delta}(t)dt)^{\frac{2}{p-1}}.
\end{eqnarray}
Then $T$ must satisfy that $T\leq 2T_1$ provided that $K\geq K_0$.
\end{lemma}
\begin{proof}
Arguing by contradiction, assume that we have $t>2T_1$. It follows from (\ref{E2-2})-(\ref{E2-3}) that
\begin{eqnarray}\label{E2-6}
G'(t)\geq G'(0)>0,~~G(t)\geq G'(0)t+G(0)\geq G(0)>0,~\forall t>0.
\end{eqnarray}
Note that $b$ is an increasing positive smooth function.
Multiplying (\ref{E2-2}) by $G'(t)$ and integrating it over $[0,t]$, we have
\begin{eqnarray*}
\frac{1}{2}G'(t)^2&\geq& A\int_0^tb^{-1}(s+R)G^p(s)dG(s)+\frac{1}{2}G'(0)^2\nonumber\\
&>&\frac{A}{(p+1)b(t+R)}(G(t)^{p+1}-G(0)^{p+1})\nonumber\\
&\geq&\frac{A}{(p+1)b(t+R)}G(t)^{p}(G(t)-G(0)),~~\forall t>0.
\end{eqnarray*}
Restricting the time interval to $t\geq\frac{G(0)}{G'(0)}$ and making use of (\ref{E2-6}), we have
\begin{eqnarray*}
\frac{1}{2}G(t)-G(0)\geq\frac{1}{2}(G'(0)t-G(0))>0.
\end{eqnarray*}
Thus we get
\begin{eqnarray*}
G(t)'>(\frac{A}{p+1})^{\frac{1}{2}}\frac{G^{\frac{p+1}{2}}(t)}{b^{\frac{1}{2}}(t+R)},~~\forall t\geq\frac{G(0)}{G'(0)},
\end{eqnarray*}
which combines with (\ref{E2-1}), for any $\delta\in(0,\frac{p-1}{2})$, we find
\begin{eqnarray}\label{E2-7}
\frac{G(t)'}{G(t)^{1+\delta}}&>&(\frac{A}{p+1})^{\frac{1}{2}}\frac{G^{\frac{p-1}{2}-\delta}(t)}{b^{\frac{1}{2}}(t+R)}\nonumber\\
&\geq&(\frac{A}{p+1})^{\frac{1}{2}}\frac{K^{\frac{p-1}{2}-\delta}}{b^{\frac{1}{2}}(t+R)a^{\delta-\frac{p-1}{2}}(t)},~~\forall t\geq T_1\geq R.
\end{eqnarray}
Integrating (\ref{E2-7}) over $[2T_1,t]$, we have
\begin{eqnarray*}
\delta^{-1}(G(2T_1)^{-\delta}-G(t)^{-\delta})>(\frac{A}{p+1})^{\frac{1}{2}}K^{\frac{p-1}{2}-\delta}\int_{2T_1}^tb(s+R)^{-\frac{1}{2}}a(s)^{\frac{p-1}{2}-\delta}ds,
\end{eqnarray*}
which gives that
\begin{eqnarray*}\label{E2-8R}
\delta^{-1}G(2T_1)^{-\delta}>(\frac{A}{p+1})^{\frac{1}{2}}K^{\frac{p-1}{2}-\delta}\int_{2T_1}^tb(s+R)^{-\frac{1}{2}}a(s)^{\frac{p-1}{2}-\delta}ds.
\end{eqnarray*}
Then making use of (\ref{E2-1}) with $t=2T_1$,
\begin{eqnarray*}
K^{-\delta}\geq(\frac{a(2T_1)}{G(2T_1)})^{\delta}>\delta(\frac{A}{p+1})^{\frac{1}{2}}a^{\delta}(2T_1)K^{\frac{p-1}{2}-\delta}\int_{2T_1}^tb(s+R)^{-\frac{1}{2}}a(s)^{\frac{p-1}{2}-\delta}ds,
\end{eqnarray*}
which implies that
\begin{eqnarray}\label{ff1}
0<K<\delta^{\frac{2}{p-1}}(\frac{A}{p+1})^{^{\frac{1}{p-1}}}a^{^{\frac{2\delta}{p-1}}}(2T_1)(\int_{2T_1}^tb(s+R)^{-\frac{1}{2}}a(s)^{\frac{p-1}{2}-\delta}ds)^{-\frac{2}{p-1}}.
\end{eqnarray}
On the other hand, by (\ref{E2-4}), we know that there exists $t^{**}> 2T_1$ such that
\begin{eqnarray}\label{ffE2-9}
\delta^{\frac{2}{p-1}}(\frac{A}{p+1})^{\frac{1}{p-1}}K_0^{-1}a^{\frac{2\delta}{p-1}}(2T_1)\leq(\int_{2T_1}^{t^{**}}b(t+R)^{-\frac{1}{2}}a^{\frac{p-1}{2}-\delta}(t)dt)^{\frac{2}{p-1}}.
\end{eqnarray}
So by (\ref{ff1})-(\ref{ffE2-9}), we get
\begin{eqnarray*}
0<K<K_0.
\end{eqnarray*}
This inequality contradicts to the choice of $K\geq K_0$. Therefore we conclude that $T\leq2T_1$.
This completes the proof.
\end{proof}

\begin{remark}
We know that Kato's Lemma (see \cite{Kato}) gives an important ODE blow up criterion. The key second inequality in Kato's Lemma is 
\begin{eqnarray*}
u''\geq bt^{-1-p}u^p,
\end{eqnarray*}
with $p>1$, $b>0$ and $t$ large.
But this result can not apply to some special decreasing inequalities, $e.g.$ $u''\geq be^{-Mt}u^p$ with $M>0$, $p>1$, $b>0$ and $t$ large. Yagdjian \cite{Yag3} gave a more generalized Kato's Lemma which can be used in a large class of finite time blow up problems, $e.g.$ the finite time blow up of solutions for the semilinear Klein-Gordon equations in de Sitter spacetime (see \cite{Yag3}). Comparing with our generalized ODE blow up criterion, the main difference between two results is that Lemma 2.1 tells us the relationship between the exponent $a_1$, $b_1$ and $p$. This is key point to prove the finite time blow up of solutions for the semilinear wave equation in de Sitter spacetime, $i.e.$ the Strauss conjecture in de Sitter spacetime. In fact, if we give the exact form of the functions $a(t)$ and $b(t)$ in Lemma 2.2, then the relationship between the exponent $a_1$, $b_1$ and $p$ ensures (\ref{E2-4}).
\end{remark}

\begin{remark}
Lemma 2.2 tells us that if we get a $\mathbb{C}^2$ function $G(t)$ which is not a polynomial increasing function, then we can also obtain a blow up criterion. We think that this result can be applied to show the blow up of solutions for a kind of non-autonomous nonlinear wave equations, for example, the Tricomi-type equation.
\end{remark}

Using the idea of Yordanov and Zhang \cite{Yorda1}, we introduce 
\begin{eqnarray*}
\phi_1(x)=\int_{\mathbb{S}^{n-1}}e^{x\cdot\omega}d\omega\geq0,
\end{eqnarray*}
which is a solution of
\begin{eqnarray*}
&&\triangle\phi_1(x)=\phi_1(x).
\end{eqnarray*}
Then, one can verify $\phi_1(x)$ (see \cite{zhou3}) such that
\begin{eqnarray}
\label{E2-101}
&&0<\phi_1(x)\leq Ce^{|x|}(1+|x|)^{-\frac{n-1}{2}},~~n\geq2,\\
\label{E2-102}
&&\phi_1(x)\sim C_ne^{|x|}|x|^{-\frac{n-1}{2}}~~as~~|x|\longrightarrow\infty,
\end{eqnarray}
where $C$ is a positive constant.

Moreover, we introduce a test function
\begin{eqnarray*}
\psi_1(t,x)=e^{-t}\phi_1(x).
\end{eqnarray*}
It is easy to see
\begin{eqnarray*}
\triangle\psi_1(t,x)=\psi_1(t,x).
\end{eqnarray*}
One can see \cite{Yorda1,Yorda2,zhou3} for more details.

Furthermore, we have the following result, which is taken from the paper of Yordanov and Zhang \cite{Yorda1}.
\begin{lemma}
Let $p>1$. $\phi_1(t)$ satisfies (\ref{E2-101})-(\ref{E2-102}). Then
\begin{eqnarray*}\label{E2-103}
\int_{|x|\leq t+1}(\psi_1(t,x))^{\frac{p}{p-1}}dx\leq C(1+t)^{n-1-\frac{(n-1)p}{2(p-1)}},~~\forall t>0.
\end{eqnarray*}
\end{lemma}

\section{Proof of Theorem 1.1}\setcounter{equation}{0}
Define
\begin{eqnarray*}
G(t)=\int_{\mathbb{R}^n}u(t,x)dx.
\end{eqnarray*}
Integrating (\ref{E1-1R1}), 
\begin{eqnarray}\label{E3-3}
\frac{d^2}{dt^2}G(t)=\partial_{tt}\int_{\mathbb{R}^n}u(t,x)dx=e^{-\frac{n}{2}(p-1)Ht}\int_{\mathbb{R}^n}|u|^p(t,x)dx.
\end{eqnarray}
Applying the H$\ddot{o}$lder inequality to the right hand side of (\ref{E3-3}), we have
\begin{eqnarray}\label{E3-3RE}
e^{-\frac{n}{2}(p-1)Ht}\int_{\mathbb{R}^n}|u|^p(t,x)dx&\geq& e^{-\frac{n}{2}(p-1)Ht}|\int_{\mathbb{R}^n}|u(t,x)dx|^p(\int_{|x|\leq t+1}dx)^{1-p}\nonumber\\
&\geq& Ce^{-\frac{n}{2}(p-1)Ht}vol(\textbf{B}^n)(t+1)^{-n(p-1)}|\int_{\mathbb{R}^n}u(t,x)dx|^p.~~~~~~~
\end{eqnarray}
Note that $\sup_{a,b>0}a^be^{-a}=(be^{-1})^b$ and $0<H<2$. It follows from (\ref{E3-3RE}) that
\begin{eqnarray}\label{E3-4}
e^{-\frac{n}{2}(p-1)Ht}\int_{\mathbb{R}^n}|u|^p(t,x)dx
\geq C'e^{-n(p-1)t}|\int_{\mathbb{R}^n}|u(t,x)dx|^p,~~
\end{eqnarray}
where $C'=Ce^{-1}(n^{-1}(p-1)^{-1}e)^{n(p-1)}vol(\textbf{B}^n)[n(p-1)(1-\frac{H}{2})]^{n(p-1)}$.

Thus by (\ref{E3-3}) and (\ref{E3-4}), we get
\begin{eqnarray}\label{E3-5}
G''(t)\geq C'e^{-n(p-1)t}|G(t)|^p.
\end{eqnarray}
On the other hand, applying the H$\ddot{o}$lder inequality to (\ref{E1-1R1}) and by Lemma 2.3, 
\begin{eqnarray}\label{E3-3R}
G''(t)&=&e^{-\frac{n}{2}(p-1)Ht}\int_{\mathbb{R}^n}|u|^p(t,x)dx\nonumber\\
&\geq& e^{-\frac{n}{2}(p-1)Ht}|\int_{\mathbb{R}^n}|u(t,x)\psi_1(t,x)dx|^p(\int_{|x|\leq t+1}(\psi_1(t,x))^{\frac{p}{p-1}}dx)^{1-p}\nonumber\\
&\geq& Ce^{-\frac{n}{2}(p-1)Ht}(1+t)^{(n-1)(1-\frac{p}{2})}\int_{\mathbb{R}^n}|u(t,x)\psi_1(t,x)dx|^pdx\nonumber\\
&=&C(1+t)^{(n-1)(1-\frac{p}{2})}\int_{\mathbb{R}^n}|e^{-\frac{n}{2p}(p-1)Ht}u(t,x)\psi_1(t,x)dx|^pdx.
\end{eqnarray}
Next we estimate the right hand side of the inequality (\ref{E3-3R}). Note that $0<H<2$.
Multiplying both side of (\ref{E1-1R1}) by the test function $\psi_2(t,x)=e^{-(\frac{n}{2p}(p-1)H+1)t}
\phi(x)$ and integrating by parts over $\mathbb{R}^n$, we have
\begin{eqnarray}\label{E3-6}
\int_{\mathbb{R}^n}\psi_2(u_{tt}-e^{-2Ht}u)dx=e^{-\frac{n}{2}(p-1)Ht}\int_{\mathbb{R}^n}\psi_2|u|^p.
\end{eqnarray}
Since
\begin{eqnarray}\label{E3-7}
\frac{d}{dt}\int_{\mathbb{R}^n}\psi_2u_tdx=\int_{\mathbb{R}^n}[\psi_2u_{tt}-(\frac{n}{2p}(p-1)H+1)\psi_2u_t]dx,
\end{eqnarray}
\begin{eqnarray}
\label{E3-8}
(\frac{n}{2p}(p-1)H&+&1)e^{-[(\frac{n}{2p}(p-1)H+1)t+\frac{1}{2(\frac{n}{2p}(p-1)H+1)H}e^{-2Ht}]}\nonumber\\
&&\times\frac{d}{dt}\int_{\mathbb{R}^n}(e^{[(\frac{n}{2p}(p-1)H+1)t+\frac{1}{2(\frac{n}{2p}(p-1)H+1)H}e^{-2Ht}]}\psi_2u)dt\nonumber\\
&=&\int_{\mathbb{R}^n}((\frac{n}{2p}(p-1)H+1)\psi_2u_t-e^{-2Ht}\psi_2u)dx,
\end{eqnarray}
summing up (\ref{E3-7})-(\ref{E3-8}), by noticing (\ref{E3-6}), we get
\begin{eqnarray}\label{E3-9}
\frac{d}{dt}\int_{\mathbb{R}^n}\psi_2u_tdx&+&
(\frac{n}{2p}(p-1)H+1)e^{-[(\frac{n}{2p}(p-1)H+1)t+\frac{1}{2(\frac{n}{2p}(p-1)H+1)H}e^{-2Ht}]}\nonumber\\
&&\times\frac{d}{dt}\int_{\mathbb{R}^n}(e^{[(\frac{n}{2p}(p-1)H+1)t+\frac{1}{2(\frac{n}{2p}(p-1)H+1)H}e^{-2Ht}]}\psi_2u)dt\nonumber\\
&=&\int_{\mathbb{R}^n}(\psi_2u_{tt}-e^{-2Ht}\psi_2u)dx\nonumber\\
&=&e^{-\frac{n}{2}(p-1)Ht}\int_{\mathbb{R}^n}\psi_2|u|^pdx.~~~~~~~~~
\end{eqnarray}
It follows from integrating (\ref{E3-9}) on $[0,t]$ that
\begin{eqnarray}\label{E3-10}
\int_{\mathbb{R}^n}\psi_2[u_t+(\frac{n}{2p}(p-1)H+1)u]dx&+&\int_0^t\int_{\mathbb{R}^n}[(\frac{n}{2p}(p-1)H+1)^2-e^{-2Ht}]\psi_2udxdt\nonumber\\
&=&\int_{\mathbb{R}^n}\phi_1[u_1+(\frac{n}{2p}(p-1)H+1)u_0]dx\nonumber\\
&&+\int_0^t\int_{\mathbb{R}^n}e^{-\frac{n}{2}(p-1)Ht}\psi_2|u|^pdx.~~~~~~~~
\end{eqnarray}
Let
\begin{eqnarray*}
G_1(t)=\int_{\mathbb{R}^n}\psi_2(t,x)u(t,x)dx.
\end{eqnarray*}
Since
\begin{eqnarray*}
&&e^{-\frac{n}{2}(p-1)Ht}\psi_2|u|^p>0,\\
&&\int_{\mathbb{R}^n}\phi_1[u_1+(\frac{n}{2p}(p-1)H+1)u_0]dx>0,
\end{eqnarray*}
the equality (\ref{E3-10}) gives that
\begin{eqnarray*}
\int_{\mathbb{R}^n}\psi_2[u_t+(\frac{n}{2p}(p-1)H+1)u]dx+\int_0^t\int_{\mathbb{R}^n}[(\frac{n}{2p}(p-1)H+1)^2-e^{-2Ht}]\psi_2udxdt\geq0,
\end{eqnarray*}
which means that
\begin{eqnarray}\label{E3-11}
\frac{d}{dt}G_1(t)+2(\frac{n}{2p}(p-1)H+1)G_1(t)+\int_0^t\int_{\mathbb{R}^n}[(\frac{n}{2p}(p-1)H+1)^2-e^{-2Ht}]\psi_2udxdt
\geq0.~~~~~
\end{eqnarray}
Note that $0<H<2$ and
\begin{eqnarray*}
G'_1(0)+2(\frac{n}{2p}(p-1)H+1)G_1(0)=\int_{\mathbb{R}^n}\phi_1[u_1+(\frac{n}{2p}(p-1)H+1)u_0]\geq0.
\end{eqnarray*}
Thus if we can prove the second order differential inequality
\begin{eqnarray}\label{E3-12}
\frac{d^2}{dt^2}G_1(t)+2(\frac{n}{2p}(p-1)H+1)\frac{d}{dt}G_1(t)+[(\frac{n}{2p}(p-1)H+1)^2-e^{-2Ht}]G_1(t)
\geq0,~~~
\end{eqnarray}
then (\ref{E3-11}) holds.

Note that (\ref{E3-12}) is a Riccati-type differential inequality. It is difficult to solve it. But
we know that $G_1(t)>0$ by noticing (\ref{E2-102}) and $G(t)>0$. To see (\ref{E3-12}), it is sufficient to show
\begin{eqnarray}\label{E3-13}
\frac{d^2}{dt^2}G_1(t)+2(\frac{n}{2p}(p-1)H+1)\frac{d}{dt}G_1(t)\geq0.
\end{eqnarray}
Direct computation of (\ref{E3-13}), we have
\begin{eqnarray}\label{E3-14}
\frac{d}{dt}G_1(t)+2(\frac{n}{2p}(p-1)H+1)G_1(t)\geq G_1'(0)+2(\frac{n}{2p}(p-1)H+1)G_1(0),
\end{eqnarray}
where $G_1'(0)=\int_{\mathbb{R}^n}\phi_1(u_1-u_0)dx$ and $G_1(0)=\int_{\mathbb{R}^n}\phi_1u_0dx$.

We multiply (\ref{E3-14}) by $e^{2(\frac{n}{2p}(p-1)H+1)t}$ and integrate on $[0,t]$, we get
\begin{eqnarray*}
G_1(t)&\geq&\frac{1}{2}(1-e^{-2(\frac{n}{2p}(p-1)H+1)t})\int_{\mathbb{R}^n}\phi_1[u_1+(\frac{n}{2p}(p-1)H+1)u_0]\nonumber\\
&&+e^{-(\frac{n}{2p}(p-1)H+1)t}\int_{\mathbb{R}^n}\phi_1u_0dx\nonumber\\
&\geq&C>0,
\end{eqnarray*}
which combining with (\ref{E3-3R}) gives that
\begin{eqnarray}\label{E3-15}
G''(t)\geq C(1+t)^{(n-1)(1-\frac{p}{2})}.
\end{eqnarray}
Hence, integrating (\ref{E3-15}) twice on $[0,t]$, we obtain
\begin{eqnarray}\label{E3-16}
G(t)\geq C(1+t)^{(n-1)(1-\frac{p}{2})+2},
\end{eqnarray}
where we use $G'(0)>0$ and $G(0)>0$.

Let
\begin{eqnarray*}
&&a_1=(n-1)(1-\frac{p}{2})+2,\\
&&b_1=n(p-1),\\
&&A=Ce^{-1}(n^{-1}(p-1)^{-1}e)^{n(p-1)}vol(\textbf{B}^n)[n(p-1)(1-\frac{H}{2})]^{n(p-1)}.
\end{eqnarray*}
Then
\begin{eqnarray*}
b_1-a_1(p-1)&=&n(p-1)-[(n-1)(1-\frac{p}{2})+2](p-1)\nonumber\\
&=&(p-1)[(n-1)\frac{p}{2}-1]<2,~~\forall~p>1.
\end{eqnarray*}
It is easy to see that the solution set is $p\in(1,p_c(n))$.
Hence, by (\ref{E3-5}) and (\ref{E3-16}), we can apply Lemma 2.1 to get that $G(t)$ will blow up in finite time, then the solution to problem (\ref{E1-1R1}) will blow up in finite time.
At last, we estimate the lifespan result.
Since $G''(t)\geq0$ and $G'(0)\geq0$, $G(t)$ is an increasing smooth function.
So it holds
\begin{eqnarray*}
G(t)=\int_{\mathbb{R}^n}u(t,x)dx\geq\epsilon\int_{\mathbb{R}^n}f(x)dx\geq C\epsilon.
\end{eqnarray*}
Using (\ref{E2-ERR}) in Lemma 2.1, we have
\begin{eqnarray*}
T(\epsilon)&\leq&C_0\epsilon^{-\frac{(p-1)}{(p-1)a_1-b_1+2}}\\
&\leq& C_0\epsilon^{-\frac{p-1}{(p-1)[1-(n-1)\frac{p}{2}]+2}},
\end{eqnarray*}
where $2^{-\frac{(p-1)[1-(n-1)\frac{p}{2}]+2}{p-1}}\leq\epsilon\leq1$ and $C_0$ is a positive constant depending on $A$ but independent of $\epsilon$.
We complete the proof of Theorem 1.1.

\section{Proof of Theorem 1.2}\setcounter{equation}{0}
Define
\begin{eqnarray*}
G(t)=\int_{\mathbb{R}^n}u(t,x)dx.
\end{eqnarray*}
Integrating (\ref{E1-002}), we derive
\begin{eqnarray}\label{E4-1}
\frac{d^2}{dt^2}G(t)=\partial_{tt}\int_{\mathbb{R}^n}u(t,x)dx\geq e^{-\frac{n}{2}(p-1)Ht}\int_{\mathbb{R}^n}|\partial_tu|^p(t,x)dx,
\end{eqnarray}
and
\begin{eqnarray}\label{E4-2}
\frac{d^2}{dt^2}G(t)\geq e^{-\frac{n}{2}(p-1)Ht}\int_{\mathbb{R}^n}|\nabla u|^p(t,x)dx\geq e^{-\frac{n}{2}(p-1)Ht}\int_{\mathbb{R}^n}|u|^p(t,x)dx.
\end{eqnarray}
Then applying the H$\ddot{o}$lder inequality to the right hand side of (\ref{E4-1}) and using $\sup_{a,b>0}a^be^{-a}=(be^{-1})^b$, we get
\begin{eqnarray}\label{E4-3}
G''(t)&\geq&
e^{-\frac{n}{2}(p-1)Ht}\int_{\mathbb{R}^n}|\partial_tu|^p(t,x)dx\nonumber\\
&\geq& e^{-\frac{n}{2}(p-1)Ht}|\int_{\mathbb{R}^n}|\partial_tu(t,x)dx|^p(\int_{|x|\leq t+1}dx)^{1-p}\nonumber\\
&\geq& Ce^{-\frac{n}{2}(p-1)Ht}vol(\textbf{B}^n)(t+1)^{-n(p-1)}|\int_{\mathbb{R}^n}|\partial_tu(t,x)dx|^p\nonumber\\
&\geq&C'e^{-n(p-1)Ht}|\int_{\mathbb{R}^n}\partial_tu(t,x)dx|^p\nonumber\\
&=&C'e^{-n(p-1)Ht}|G'(t)|^p,
\end{eqnarray}
where $C'=Ce^{-1}(n^{-1}(p-1)^{-1}e)^{n(p-1)}vol(\textbf{B}^n)(\frac{nH(p-1)}{2})^{n(p-1)}$.

It follows from (\ref{E4-3}) that
\begin{eqnarray}\label{E4-4}
\frac{1}{1-p}\frac{d}{dt}|G'(t)|^{1-p}&\geq&C'e^{-nH(p-1)t}.
\end{eqnarray}
Integrating (\ref{E4-4}) over $[0,t]$,
\begin{eqnarray}\label{E4-5}
|G'(t)|&\geq&\frac{C'}{n}e^{nHt}.
\end{eqnarray}
It follows from (\ref{E4-2}) that $G'(t)=\int_{\mathbb{R}^n}\partial_tu(t,x)dx$ is an increasing function for $t\geq0$. Since $g(x)\geq0$ in (\ref{E1-3}), $G(t)$ is also an increasing function for $t\geq0$. By $f(x)\geq0$ in (\ref{E1-3}), we konw that $G(t)>0$. Thus it follows from (\ref{E4-5}) and $\sup_{a,b>0}a^be^{-a}=(be^{-1})^b$ that
\begin{eqnarray}\label{E4-6}
G(t)&\geq&\frac{C'}{n^2}e^{nHt}\geq\frac{C'}{n^2}(\frac{1}{2e})^{-\frac{Hn}{2}}t^{\frac{Hn}{2}}.
\end{eqnarray}
On the other hand, by the H$\ddot{o}$lder inequality, we derive
\begin{eqnarray*}
G(t)&=&\int_{\mathbb{R}^n}u(t,x)dx\nonumber\\
&\leq&(\int_{\mathbb{R}^n}|u(t,x)|^{p})^{\frac{1}{p}}(\int_{|x|\leq1+t}dx)^{1-\frac{1}{p}}\nonumber\\
&\leq&C(1+t)^{n(1-\frac{1}{p})}(\int_{\mathbb{R}^n}|u(t,x)|^{p})^{\frac{1}{p}},
\end{eqnarray*}
which implies that
\begin{eqnarray*}
C\int_{\mathbb{R}^n}|u(t,x)|^{p}dx\geq (1+t)^{-n(p-1)}G^{p}(t).
\end{eqnarray*}
Then by (\ref{E4-2}) and $\sup_{a,b>0}a^be^{-a}=(be^{-1})^b$, we obtain
\begin{eqnarray}\label{E4-7}
G''(t)&\geq&Ce^{-\frac{n}{2}(p-1)Ht} (1+t)^{-n(p-1)}G^{p}(t)\nonumber\\
&\geq&C(n(p-1)e^{-1})^{n(p-1)}e^{-1}e^{-(\frac{n}{2}(p-1)H+1)t}G^{p}(t).
\end{eqnarray}
Let
\begin{eqnarray*}
&&a_1=\frac{nH}{2},\\
&&b_1=\frac{n(p-1)H}{2}+1,\\
&&A=\frac{C'}{n^2}(\frac{1}{2e})^{-\frac{Hn}{2}}.
\end{eqnarray*}
Then direct computation shows that
\begin{eqnarray*}
b_1-a_1(p-1)=\frac{n(p-1)H}{2}+1-\frac{nH(p-1)}{2}=1<2,
\end{eqnarray*}
where
\begin{eqnarray*}
1<p<1+\frac{2}{n-1}.
\end{eqnarray*}
It is easy to see that the solution set is $p\in(1,p'_c(n))$.
Hence, by (\ref{E4-6})-(\ref{E4-7}), we can apply Lemma 2.1 to get that $G(t)$ will blow up in finite time, then the solutions to problem (\ref{E1-002}) will blow up in finite time.
Using (\ref{E2-ERR}) in Lemma 2.1, we obtain
\begin{eqnarray*}
T(\epsilon)&\leq& C_0\epsilon^{-\frac{(p-1)}{(p-1)a_1-b_1+2}},\\
&\leq& C_0\epsilon^{-(p-1)},
\end{eqnarray*}
where $2^{-\frac{1}{p-1}}\leq\epsilon\leq1$ and $C_0$ is a positive constant depending on $A$ but independent of $\epsilon$. This completes the proof of Theorem 1.2.

\section*{Acknowledgements}
The author expresses his sincere thanks to the anonymous
referees for very careful reading and for providing many valuable
comments and suggestions which led to improvement of this paper.
The author expresses his sincerely thanks to Prof. Y. Zhou for his discussion and suggestion. The author also expresses his sincerely thanks to Prof. K. Yagdjian for his interest on this paper, his suggestion and sending us the interesting papers of \cite{Yag5,Yag3,Yag4}. The author is also expresses his sincerely thanks to Dr. N.A. Lai for his suggestion.
This work is supported by the Fundamental Research Funds for the Central Universities of Xiamen University, No. 20720150013.

\end{document}